\input ./style/arxiv-vmsta.cfg
\documentclass[numbers,compress,v1.0.1]{vmsta}
\usepackage{mathbh}

\volume{4}% Updated by VTEXPTS2LaTeX.exe, 30.06.2017 12:59
\issue{2}% Updated by VTEXPTS2LaTeX.exe, 30.06.2017 12:59
\pubyear{2017}
\firstpage{93}% Updated by VTEXPTS2LaTeX.exe, 30.06.2017 12:59
\lastpage{108}% Updated by VTEXPTS2LaTeX.exe, 30.06.2017 12:59
\doi{10.15559/17-VMSTA76}% Updated by VTEXPTS2LaTeX.exe, 10.04.2017
\startlocaldefs

\newcommand{\rrvert}{\vert}
\newcommand{\llvert}{\vert}
\allowdisplaybreaks

\newtheorem{thm}{Theorem}
\newtheorem{proposition}{Proposition}
\newtheorem{lemma}{Lemma}

\theoremstyle{definition}
\newtheorem{remark}{Remark}

\theoremstyle{remark}
\newtheorem{Case}{Case}

\newcommand{\E}{\mathbb E}

\newcommand{\R}{\mathbb{R}}
\newcommand{\N}{\mathbb{N}}

\newcommand{\mmp}{\mathbb{P}}

\newcommand{\eqdistr}{\stackrel{{\rm d}}{=}}

\newcommand{\todistrfdd}{\stackrel{{\rm f.d.}}{\Longrightarrow}}

\DeclareMathOperator{\1}{\mathbh{1}}
\endlocaldefs

\begin{document}
\begin{frontmatter}

\title{A functional limit theorem for random processes with immigration
in the case of heavy tails}
%\author[]{\inits{}\fnm{}\snm{}\corref{cor1}}\email{}
%\cortext[cor1]{Corresponding author.}
%
%\author[]{\inits{}\fnm{}\snm{}}\email{}
%
%%\fnref{f1}
%%\fntext[]{Some remarks}
%
%\address[]{}
%\address[]{}

\author[a]{\inits{A.}\fnm{Alexander}\snm{Marynych}\corref{cor1}}\email
{marynych@unicyb.kiev.ua}
\cortext[cor1]{Corresponding author.}
\author[b]{\inits{G.}\fnm{Glib}\snm{Verovkin}}\email{glebverov@gmail.com}

\address[a]{Faculty of Computer Science and Cybernetics, Taras
Shevchenko National University of Kyiv, 01601 Kyiv, Ukraine}
\address[b]{Faculty of Mechanics and Mathematics, Taras Shevchenko
National University~of~Kyiv, 01601 Kyiv, Ukraine}

\markboth{A. Marynych, G. Verovkin}{A functional limit theorem for
random processes with immigration}
%\markboth{A. Authors}{Title}

%\begin{abstract}
%\end{abstract}

\begin{abstract}
Let $(X_k,\xi_k)_{k\in\N}$ be a sequence of independent copies of a
pair $(X,\xi)$ where $X$ is a random process with paths in the
Skorokhod space $D[0,\infty)$ and $\xi$ is a positive random variable.
The random process with immigration $(Y(u))_{u\in\R}$ is defined as the
a.s. finite sum $Y(u)=\sum_{k\geq0}X_{k+1}(u-\xi_1-\cdots-\xi_k)\1_{\{
\xi_1+\cdots+\xi_k\leq u\}}$.
We obtain a functional limit theorem for the process $(Y(ut))_{u\geq
0}$, as $t\to\infty$, when the law of $\xi$ belongs to the domain of
attraction of an $\alpha$-stable law with $\alpha\in(0,1)$, and the
process $X$ oscillates moderately around its mean $\E[X(t)]$. In this
situation the process $(Y(ut))_{u\geq0}$,
%after an appropriate scaling
when scaled appropriately,
converges weakly in the Skorokhod space $D(0,\infty)$ to a fractionally
integrated inverse stable subordinator.
\end{abstract}



%\begin{keywords}
%\kwd{}
%\kwd{}
%\kwd{}
%\kwd{}
%\end{keywords}
%\begin{keywords}[2010]% [PACS], [JEL]
%\kwd{}
%\kwd{}
%\kwd{}
%\kwd{}
%\end{keywords}

%
\received{31 January 2017}% Updated by VTEXPTS2LaTeX.exe, 10.04.2017
%10:32
%
\revised{26 March 2017}% Updated by VTEXPTS2LaTeX.exe, 10.04.2017 10:32
\accepted{27 March 2017}% Updated by VTEXPTS2LaTeX.exe, 10.04.2017 10:32
\publishedonline{12 April 2017}
\end{frontmatter}

\section{Introduction and main result}
Let $(X_k,\xi_k)_{k\in\N}$ be a sequence of independent copies of a
pair $(X,\xi)$ where $X$ is a random process with paths in $D[0,\infty
)$ and $\xi$ is a positive random variable. We impose no conditions on
the dependence structure of $(X,\xi)$.
%Here and
Hereafter $\N_0$ denotes the set of non-negative integers $\{
0,1,2,\ldots\}$.

Let $(S_n)_{n\in\N_0}$ be a standard zero-delayed random walk:
\begin{equation}
\label{eq:random_walk_def} S_0:=0,\qquad S_n:=\xi_1+
\cdots+\xi_n,\quad n\in\N,
\end{equation}
and let $(\nu(t))_{t\in\R}$ be the corresponding first-passage time
process for $(S_n)_{n\in\N_0}$:
\[
\nu(t):=\inf\{k\in\N_0:S_k > t\},\quad t\in\R.
\]
The {\it random process with immigration} $Y=(Y(u))_{u\in\R}$ is
defined as a finite sum
\[
Y(u):=\sum_{k \geq0}X_{k+1}(u-S_k)
\1_{\{S_k\leq u\}}=\sum_{k=0}^{\nu
(u)-1}X_{k+1}(u-S_k),
\quad u\in\R.
\]
This family of random processes was introduced in \cite
{Iksanov+Marynych+Meiners:2017-1} as a generalization of several known
objects in applied probability including branching processes with
immigration %(corresponding to
(in case of $X$ being a branching process) and renewal shot noise
processes (in case of $X(t)=h(t)$ a.s. for some $h\in D[0,\infty)$).
The process $X$ is usually called {\it a response process}, or {\it a
response function} if $X(t)=h(t)$ a.s. for some deterministic function $h$.

The problem of weak convergence of random processes with immigration
was addressed in \cite
{Iksanov+Marynych+Meiners:2017-1,Iksanov+Marynych+Meiners:2017-2,Marynych:2015}
where the authors give a more or less complete picture of the weak
convergence of finite-dimensional distributions of $(Y(ut))_{u\geq0}$
or $(Y(u+t))_{u\in\R}$, as $t\to\infty$. The case of renewal shot noise
process has received much attention in the past years, see \cite
{Iksanov:2013,Iksanov+Kabluchko+Marynych:2016-1,Iksanov+Marynych+Meiners:2014,Iksanov+Kabluchko+Marynych+Shevchenko:2017}.
A comprehensive survey of the subject is given in Chapter 3 of the
recent book \cite{Iksanov_book:2016}.

A much more delicate question of weak convergence of $Y$ in functional
spaces, to the best of our knowledge, was only investigated either for
particular response processes, or in the simple case %where
when $\xi$ is exponentially distributed. In the latter situation $Y$ is
called {\it a Poisson shot noise process}. In the list below $\eta$ is
a random variable which satisfies certain assumptions specified in the
corresponding papers:
\begin{itemize}
\item if $\xi$ has exponential distribution and either $X(t)=\1_{\{\eta
> t\}}$ or $X(t)=t\wedge\eta$, functional limit theorems for $Y$ were derived
in \cite{Resnick+Rootzen:2000};
\item if $X(t)=\1_{\{\eta> t\}}$ and $\E\xi<\infty$, a functional
limit theorem for $Y$ was established in \cite{Iksanov+Jedidi+Bouzeffour:2017+};

\item if $X(t)=\1_{\{\eta\leq t\}}$, functional limit theorems for $Y$
are given in \cite{Alsmeyer+Iksanov+Marynych:2016};

\item if $\xi$ has exponential distribution and $X(t)=\eta f(t)$ for
some deterministic function $f$, limit theorems for $Y$ were obtained
in \cite{Kluppelberg+Kuhn:2004};

\item in \cite{Iksanov+Marynych+Meiners:2017-2,Marynych:2015}
sufficient conditions for weak convergence of $(Y(u+t))_{u\in\R}$ to a
stationary process with immigration were found.
\end{itemize}

In this paper we treat the case where $\xi$ is heavy-tailed, more
precisely we assume that
\begin{equation}
\label{eq:tail_xi_reg_var} \mmp\{\xi>t\}\sim t^{-\alpha}\ell_{\xi}(t),\quad t\to
\infty,
\end{equation}
for some $\ell_{\xi}$ slowly varying at infinity, and $\alpha\in
(0,1)$. Assuming \eqref{eq:tail_xi_reg_var}, we obtain a functional
limit theorem for a quite general class of response processes. The
class of such processes can be described by a common property: they do
not ``oscillate to much'' around the mean $\E[X(t)]$, which itself
varies regularly with parameter $\rho>-\alpha$. Let us briefly outline
our approach based on ideas borrowed from \cite
{Iksanov+Marynych+Meiners:2017-1}. Put $h(t):=\E[X(t)]$ and
write\footnote{In what follows we always assume that $h$ exists and is
a c\`{a}dl\`{a}g function.}
\begin{equation}
\label{eq:decompose} Y(t)=\sum_{k\geq0} \bigl(X_{k+1}(t-S_k)-h(t-S_k)
\bigr)\1_{\{S_k\leq t\}
}+\sum_{k\geq0}h(t-S_k)
\1_{\{S_k\leq t\}}.
\end{equation}
We investigate the two summands in the right-hand side separately. The
second summand is a standard renewal shot noise process with response
function $h$. Under condition \eqref{eq:tail_xi_reg_var} and assuming that
\begin{equation}
\label{eq:h_reg_var} h(t)=\E\bigl[X(t)\bigr] \sim t^{\rho}\ell_{h}(t),
\quad t\to\infty,
\end{equation}
for some $\rho\in\R$ and a slowly varying function $\ell_{h}$, it was
proved in \cite[Theorem 2.9]{Iksanov+Marynych+Meiners:2014} and \cite
[Theorem 2.1]{Iksanov+Kabluchko+Marynych+Shevchenko:2017} that
\begin{equation}
\label{eq:shot_noise_fidi_conv} \biggl(\frac{\mmp\{\xi>t\}}{h(t)}\sum_{k\geq0}h(ut-S_k)
\1_{\{S_k\leq
ut\}} \biggr)_{u>0}\todistrfdd \xch{\bigl(J_{\alpha,\rho}(u)\bigr)_{u>0},}{\bigl(J_{\alpha,\rho}(u)\bigr)_{u>0}}\quad t\to\infty,
\end{equation}
where $J_{\alpha,\rho}=(J_{\alpha,\rho}(u))_{u\geq0}$ is a so-called
fractionally integrated inverse $\alpha$-stable subordinator. The
process $J_{\alpha,\rho}$ is defined as a pathwise Lebesgue--Stieltjes integral
\begin{equation}
\label{eq:FIISS_def} J_{\alpha,\rho}(u)=\int_{[0,\,u]}(u-y)^{\rho}{
\rm d}W_{\alpha
}^{\leftarrow}(y),\quad u\geq0.
\end{equation}
In this formula $W_{\alpha}^{\leftarrow}(y):=\inf\{t\geq0:W_{\alpha
}(t)>y\}$, $y\geq0$, is a generalized inverse of an $\alpha$-stable
subordinator $(W_{\alpha}(t))_{t\geq0}$ with the Laplace exponent
\[
-\log\E e^{-sW_{\alpha}(1)}=\varGamma(1-\alpha)s^{\alpha},\quad s\geq0.
\]
It is also known that convergence of finite-dimensional distributions
\eqref{eq:shot_noise_fidi_conv} can be\break strengthened to convergence in
the Skorokhod space $D(0,\infty)$ endowed with the $J_1$-topology if
$\rho>-\alpha$, see Theorem 2.1 in \cite
{Iksanov+Kabluchko+Marynych+Shevchenko:2017}. If $\rho\leq-\alpha$ the
process $(J_{\alpha,\rho}(u))_{u\geq0}$, being a.s. finite for every
fixed $u\geq0$, has a.s. locally unbounded trajectories, see
Proposition 2.5 in \cite{Iksanov+Kabluchko+Marynych+Shevchenko:2017}.

Turning to the first summand in \eqref{eq:decompose} we note that it is
the a.s. limit of a martingale $(R(j,t),\mathcal{F}_j)_{j\in\N}$, where
$\mathcal{F}_j:=\sigma((X_k,\xi_k):1\leq k\leq j)$ and
\[
R(j,t):=\sum_{k=0}^{j-1}
\bigl(X_{k+1}(t-S_k)-h(t-S_k) \bigr)
\1_{\{
S_k\leq t\}},\quad j\in\N.
\]
Applying the martingale central limit theory it is possible to show
that under appropriate assumptions (which are of no importance for this paper)
\[
\biggl(\sqrt{\frac{\mmp\{\xi>t\}}{v(t)}}\sum_{k\geq0}
\bigl(X_{k+1}(ut-S_k)-h(ut-S_k) \bigr)
\1_{\{S_k\leq ut\}} \biggr)_{u>0}\todistrfdd\bigl(Z(u)\bigr)_{u>0},
\]
as $t\to\infty$, for a non-trivial process $Z$, where $v(t):=\E
[(X(t)-h(t))^2]$ is the variance of $X$, see Proposition 2.2 in \cite
{Iksanov+Marynych+Meiners:2017-1}.

We are interested in situations when the second summand in \eqref
{eq:decompose} asymptotically dominates, more precisely we are looking
for conditions ensuring
\begin{equation}
\label{eq:negligible_func} \frac{\mmp\{\xi>t\}}{h(t)}\sup_{u\in[0,\,T]} \biggl\llvert \sum
_{k\geq0} \bigl(X_{k+1}(ut-S_k)-h(ut-S_k)
\bigr)\1_{\{S_k\leq ut\}} \biggr\rrvert \overset{\mmp } {\to} 0,\quad t\to\infty,
\end{equation}
for every fixed $T>0$. From what has been mentioned above it is clear
that this can happen only if
\begin{equation}
\label{eq:negligible_fd} \lim_{t\to\infty}\frac{\mmp\{\xi>t\}v(t)}{h^2(t)}=0.
\end{equation}
Restricting our attention to the case where $v$ is regularly varying
with index $\beta\in\R$, i.e.
\begin{equation}
\label{eq:v_reg_var} v(t)\sim t^{\beta}\ell_{v}(t),\quad t\to\infty,
\end{equation}
we see that \eqref{eq:negligible_fd} holds if $\beta<\alpha+2\rho$ and
fails if $\beta>\alpha+2\rho$. As long as we do not make any
assumptions on distributional or path-wise properties of $X$ such as
e.g., monotonicity, self-similarity or independence of increments, it
can be hardly expected that condition \eqref{eq:negligible_fd} alone is
sufficient for \eqref{eq:negligible_func}. Nevertheless, we will show
that \eqref{eq:negligible_func} holds true under additional assumptions
on the asymptotic behavior of higher centered moments $\E
[(X(t)-h(t))^{2l}]$, $l=1,2,\ldots$, and an additional technical
assumption. Our first main result treats the case where the moments of
the normalized process $([X(t)-h(t)]/v(t))_{t\geq0}$ are bounded
uniformly in $t\geq0$. Denote by $(\widehat{X}(t))_{t\geq0}$ the
centered process $(X(t)-h(t))_{t\geq0}$.

\begin{thm}\label{thm:main1}
Assume that for all $t\geq0$ and $l\in\N$ we have $\E[|X(t)|^l]<\infty
$. Further, assume that the following conditions are fulfilled:
\begin{itemize}
\item[(A1)] relation \eqref{eq:tail_xi_reg_var} holds for some $\alpha
\in(0,1)$;
\item[(A2)] relation \eqref{eq:h_reg_var} holds for some $\rho>-\alpha$;
\item[(A3)] relation \eqref{eq:v_reg_var} holds for some $\beta\in
(-\alpha,\alpha+2\rho)$;
\item[(A4)] there exists $\delta>0$ such that for every $l\in\N$ the
following two conditions hold:
\begin{equation}
\label{eq:thm1_main1} \E \bigl[\widehat{X}(t)^{2l} \bigr]\leq
C_lv^l(t),\quad t\geq0,
\end{equation}
and
\begin{equation}
\label{eq:thm1_main2} \E \Bigl[\sup_{y\in[0,\delta)}\big|\widehat{X}(t)-
\widehat{X}(t-y)\1_{\{
y\leq t\}}\big|^{l} \Bigr]\leq C_l
t^{l(\rho-\varepsilon)},\quad t\geq0,
\end{equation}
for some $C_l\in(0,\infty)$ and $\varepsilon>0$.
\end{itemize}
Then, as $t\to\infty$,
\begin{equation}
\label{flt_main} \biggl(\frac{\mmp\{\xi>t\}}{h(t)}\sum_{k\geq0}X_{k+1}(ut-S_k)
\1_{\{
S_k\leq ut\}} \biggr)_{u>0}\Rightarrow \bigl(J_{\alpha,\rho}(u)
\bigr)_{u>0},
\end{equation}
weakly on $D(0,\infty)$ endowed with the $J_1$-topology.
\end{thm}

Our second main result is mainly applicable when the process $X$ is
almost surely bounded by some (deterministic) constant. We have the
following theorem.

\begin{thm}\label{thm:main2}
Assume that for all $t\geq0$ and $l\in\N$ we have $\E|X(t)|^l<\infty$
and conditions (A1), (A2) of Theorem \ref{thm:main1} are %in force.
valid.
Further, suppose that for every $l\in\N$ there exists a constant
$C_l>0$ such that
\begin{equation}
\label{eq:thm2_main1} \E\bigl[\widehat{X}(t)^{2l}\bigr]=\E\bigl[
\bigl(X(t)-h(t)\bigr)^{2l}\bigr]\leq C_lh(t),\quad t\geq0,
\end{equation}
and for some $\delta>0$ the function $t\mapsto\E[\sup_{y\in[0,\delta
)}|\widehat{X}(t)-\widehat{X}(t-y)\1_{\{y\leq t\}}|^{l}]$ is either
directly Riemann integrable or locally bounded and
\begin{equation}
\label{eq:thm2_main2} \E\Bigl[\sup_{y\in[0,\delta)}\big|\widehat{X}(t)-
\widehat{X}(t-y)\1_{\{y\leq
t\}}\big|^{l}\Bigr]=O\bigl(\mmp\{\xi>t\}
\bigr),\quad t\to\infty.
\end{equation}
Then \eqref{flt_main} holds.
\end{thm}

Obviously, our results are far from being optimal and leave a lot of
space for improvements, yet they are applicable to several models given
in the next section.

\section{Applications}
\subsection{The number of busy servers in a \texorpdfstring{$G/G/\infty
$}{G/G/infinity} queue}
Consider a $G/G/\infty$ queue with customers arriving at
$0=S_0<S_1<S_2<\cdots$. Upon arrival each customer is served
immediately by one of infinitely many idle servers and let the service
time of the $k$th customer be $\eta_{k}$, a copy of a positive random
variable $\eta$. Put $X(t):=\1_{\{\eta>t\}}$, then the random process
with immigration
\[
Y(u)=\sum_{k\geq0}\1_{\{S_k\leq u<S_{k}+\eta_{k+1}\}},\quad u\geq0,
\]
represents the number of busy servers at time $u\geq0$. The process
$(Y(u))_{u\geq0}$ may also be interpreted as the difference between
the number of visits to $[0,t]$ of the standard
random walk $(S_k)_{k\geq0}$ and the perturbed random walk $(S_k+\eta
_{k+1})_{k\geq1}$, see \cite{Alsmeyer+Iksanov+Meiners:2015}, or as the
number of active sources in a communication network, see \cite
{Mikosch+Resnick:2006,Resnick+Rootzen:2000}. An introduction to renewal
theory for perturbed random walks can be found in \cite{Iksanov_book:2016}.

Assume that \eqref{eq:tail_xi_reg_var} holds and
\begin{equation}
\label{eq:eta_reg_var} \mmp\{\eta>t\}\sim t^{\rho}\ell_{\eta}(t),\quad t
\to\infty,
\end{equation}
for some $\rho\in(-\alpha,0]$ and $\ell_{\eta}$ slowly varying at
infinity. Note that
\[
h(t)=\mmp\{\eta>t\}\sim t^{\rho}\ell_{\eta}(t),\quad t\to\infty.
\]
Moreover, for every $l\in\N$ and every $\delta>0$,
\[
\E\bigl[\widehat{X}(t)^{2l}\bigr]=\mmp\{\eta>t\}\mmp\{\eta\leq t\}
\bigl(\mmp^{2l-1}\{ \eta>t\}+\mmp^{2l-1}\{\eta\leq t\}\bigr)\leq
h(t)
\]
and
\begin{align*}
\E&\Bigl[\sup_{y\in[0,\delta)}\big|\widehat{X}(t)-\widehat{X}(t-y)
\1_{\{y\leq
t\}}\big|^{l}\Bigr]\leq2^{l-1}\E\Bigl[\sup
_{y\in[0,\delta)}\big|\widehat{X}(t)-\widehat {X}(t-y)\1_{\{y\leq t\}}\big|\Bigr]
\\
&\leq2^{l}\mmp\{\eta>t\}\1_{\{t\leq\delta\}}+2^{l}\bigl(\mmp
\{\eta>t-\delta \}-\mmp\{\eta>t\}\bigr)\1_{\{t>\delta\}}.
\end{align*}
The function on the right-hand side is directly Riemann integrable.
Indeed, we have
\begin{align*}
&\hspace{-2mm}\sum_{n\geq1}\sup_{\delta n\leq y\leq\delta(n+1)}
\bigl(\mmp\{ \eta>y-\delta\}-\mmp\{\eta>y\}\bigr)
\\
&\leq\sum_{n\geq1}\bigl(\mmp\bigl\{\eta>(n-1)\delta
\bigr\}-\mmp\bigl\{\eta>(n+1)\delta\bigr\} \bigr)=\mmp\{\eta>0\}+\mmp\{\eta>\delta
\}\leq2,
\end{align*}
and the claim follows from the remark after the definition of direct
Riemann integrability given on p.~362 in \cite{Feller:1968}.

From Theorem \ref{thm:main2} we obtain the following result,
complementing Theorem 1.2 in \cite{Iksanov+Jedidi+Bouzeffour:2017+}
that treats the case $\E\xi<\infty$.
\begin{proposition}\label{ass:gg_queue}
Assume that $(\xi,\eta)$ is a random vector with positive components
such that \eqref{eq:tail_xi_reg_var} and \eqref{eq:eta_reg_var} hold
for $\alpha\in(0,1)$ and $\rho\in(-\alpha,0]$, respectively. Let $(\xi
_k,\eta_k)_{k\in\N}$ be a sequence of independent copies of $(\xi,\eta
)$ and $(S_k)_{k\in\N_0}$ be a random walk defined by \eqref
{eq:random_walk_def}. Then
\[
\biggl(\frac{\mmp\{\xi>t\}}{\mmp\{\eta>t\}}\sum_{k\geq0}\1_{\{S_k\leq
ut<S_k+\eta_{k+1}\}}
\biggr)_{u>0}\Rightarrow \bigl(J_{\alpha,\rho
}(u) \bigr)_{u>0},
\quad t\to\infty,
\]
weakly on $D(0,\infty)$ endowed with the $J_1$-topology.
\end{proposition}
\begin{remark}
We do not assume independence of $\xi$ and $\eta$.
\end{remark}

\subsection{Shot noise processes with a random amplitude}
Assume that $X(t)=\eta f(t)$, where $\eta$ is a non-degenerate random
variable and $f:[0,\infty)\to\R$ is a fixed c\`{a}dl\`{a}g function.
The corresponding random process with immigration
\[
Y(t)=\sum_{k\geq0}\eta_{k+1}f(t-S_k)
\1_{\{S_k\leq t\}},\quad t\geq0,
\]
where $(\eta_k)_{k\in\N}$ is a sequence of independent copies of $\eta
$, may be interpreted as a renewal shot noise process in which the
common response function $f$ is scaled at a shot $S_k$ by a random
factor $\eta_{k+1}$. In case where $(\xi_k)_{k\in\N}$ have exponential
distribution and are independent of $(\eta_k)_{k\in\N}$ such processes
were used in mathematical finance as a model of stock prices with
long-range dependence in asset returns, see \cite{Kluppelberg+Kuhn:2004}.

Note that if $\E|\eta|^{l}<\infty$ for all $l\in\N$, then
\begin{align*}
&h(t)=(\E\eta)f(t),\qquad v(t)={\rm Var}(\eta)f^2(t),
\\
&\E\bigl[\bigl(X(t)-h(t)\bigr)^{2l}\bigr]=\E\bigl[(\eta-\E
\eta)^{2l}\bigr]f^{2l}(t)\leq C_l
v^l(t),\quad l\in\N,
\end{align*}
for some $C_l>0$. Assume now that $f$ varies regularly with index $\rho
>-\alpha$ and additionally satisfies
\begin{equation}
\label{eq:reg_var_f_add_ass} \sup_{y\in[0,\delta)} \bigl\llvert f(t)-f(t-y) \bigr\rrvert
= O\bigl(t^{\rho-\varepsilon
}\bigr),\quad t\to\infty,
\end{equation}
for some $\delta>0$ and $\varepsilon>0$. Then
\begin{align*}
\E\Bigl[\sup_{y\in[0,\delta)}\big|\widehat{X}(t)-\widehat{X}(t-y)
\1_{\{y\leq t\}
}\big|^{l}\Bigr]&=\E|\eta-\E\eta|^{l}\sup
_{y\in[0,\delta)}\big|f(t)-f(t-y)\1_{\{
y\leq t\}}\big|^{l}
\\
&=O\bigl(t^{l(\rho-\varepsilon)}\bigr),\quad t\to\infty.
\end{align*}
Hence, all assumptions of Theorem \ref{thm:main1} hold (if $\E\eta<0$,
Theorem \ref{thm:main1} is applicable to the process $-X$) and we have
the following result.
\begin{proposition}\label{ass:mult_shot_noise}
Assume that $\E|\eta|^{l}<\infty$ for all $n\in\N$, $\E\eta\neq0$ and
\eqref{eq:tail_xi_reg_var} holds. If $f:[0,\infty)\to\R$ satisfies
\[
f(t)\sim t^{\rho}\ell_f(t),\quad t\to\infty,
\]
for some $\rho>-\alpha$ and $\ell_f$ slowly varying at infinity, and
\eqref{eq:reg_var_f_add_ass} holds, then
\[
\biggl(\frac{\mmp\{\xi>t\}}{f(t)\E\eta}\sum_{k\geq0}\eta
_{k+1}f(ut-S_k)\1_{\{S_k\leq ut\}} \biggr)_{u>0}
\Rightarrow \bigl( J_{\alpha,\rho}(u) \bigr)_{u>0},\quad t\to\infty,
\]
weakly on $D(0,\infty)$ endowed with the $J_1$-topology.
\end{proposition}
This result complements the convergence of finite-dimensional
distributions provided by Example 3.3 in \cite{Iksanov+Marynych+Meiners:2017-1}.
\begin{remark}
In general, condition \eqref{eq:reg_var_f_add_ass} might not hold for a
function $f$ which is regularly varying with index $\rho\in\R$. %Put,
Take, for example,
\[
f(t)=1+\frac{(-1)^{[t]}}{\log[t]}\1_{\{t>1\}}.
\]
Then, $f$ is regularly varying with index $\rho=0$, but for every
$\delta>0$ and large $n\in\N$ we have
\[
\sup_{y\in[0,\delta)}\big|f(2n)-f(2n-y)\big|\geq\sup_{y\in[0,\delta\wedge
1)}\big|f(2n)-f(2n-y)\big|
\geq\frac{2}{\log(2n)}.
\]
Hence, \eqref{eq:reg_var_f_add_ass} does not hold. On the other hand,
if $f$ is differentiable with an eventually monotone derivative
$f^{\prime}$, then \eqref{eq:reg_var_f_add_ass} holds by the mean value
theorem for differentiable functions and Theorem 1.7.2 in \cite
{Bingham+Goldie+Teugels:1987}.
\end{remark}

\section{Proof of Theorems \ref{thm:main1} and \ref{thm:main2}}
The proofs of Theorems \ref{thm:main1} and \ref{thm:main2} rely on the
same ideas, so we will prove them simultaneously. Pick $\delta>0$ such
that all assumptions of Theorem \ref{thm:main1} or Theorem \ref
{thm:main2} hold. This $\delta>0$ remains fixed throughout the proof.

In view of assumptions (A1) and (A2) and the fact that $h$ is c\`{a}dl\`
{a}g we infer from Theorem 2.1 in \cite
{Iksanov+Kabluchko+Marynych+Shevchenko:2017} that
\begin{equation}
\label{eq:shot_noise_func_conv} \biggl(\frac{\mmp\{\xi>t\}}{h(t)}\sum_{k\geq0}h(ut-S_k)
\1_{\{S_k\leq
ut\}} \biggr)_{u>0}\Rightarrow\bigl(J_{\alpha,\rho}(u)
\bigr)_{u>0}\quad t\to \infty,
\end{equation}
weakly on $D(0,\infty)$ endowed with the $J_1$-topology. Note that in
Theorem 2.1 of %in
\cite{Iksanov+Kabluchko+Marynych+Shevchenko:2017} $h$ is assumed
monotone (or eventually monotone). However, this assumption is
redundant. The only places which have to be adjusted in the proofs are
two displays on p.~90, where $h(0)$ should be replaced by $\sup_{y\in
[0,c]}h(y)$.

Hence, from \eqref{eq:decompose} we see that it is enough to check, for
every fixed $T>0$, that
\begin{equation}
\label{eq:proof1} \frac{\mmp\{\xi>t\}}{h(t)}\sup_{u\in[0,T]}\big|\widetilde{Y}(ut)\big|
\overset {\mmp} {\to} 0,\quad t\to\infty,
\end{equation}
where $\widetilde{Y}(t):=\sum_{k\geq0}
(X_{k+1}(t-S_k)-h(t-S_k) )\1_{\{S_k\leq t\}}$ for $t\geq0$.
Moreover, it suffices to show that
\begin{equation}
\label{eq:proof2} \frac{\mmp\{\xi>t\}}{h(t)}\big|\widetilde{Y}(t)\big|\overset{a.s.} {\to} 0,\quad t
\to\infty.
\end{equation}
Indeed, for every fixed $s>0$,
\begin{align*}
&\frac{\mmp\{\xi>t\}}{h(t)}\sup_{u\in[0,T]}\big|\widetilde{Y}(ut)\big|\\
&\quad\leq
\frac
{\mmp\{\xi>t\}}{h(t)}\sup_{u\in[0,s]}\big|\widetilde{Y}(u)\big|+
\frac{\mmp\{\xi
>t\}}{h(t)}\sup_{u\in[s,Tt]}\big|\widetilde{Y}(u)\big|
\\
&\quad\leq\frac{\mmp\{\xi>t\}}{h(t)}\sup_{u\in[0,s]}\big|\widetilde{Y}(u)\big|+\frac{\mmp\{\xi>t\}}{h(t)}\sup_{u\in[s,Tt]}\frac
{h(u)}{\mmp\{\xi>u\}}
\sup_{u\in[s,Tt]} \biggl\llvert \frac{\mmp\{\xi>u\}
}{h(u)}\widetilde{Y}(u)
\biggr\rrvert .
\end{align*}
Since $t\mapsto h(t)/\mmp\{\xi>t\}$ is regularly varying with positive
index $\rho+\alpha$,
\[
\sup_{u\in[s,Tt]}\frac{h(u)}{\mmp\{\xi>u\}} \sim\frac{h(Tt)}{\mmp\{\xi
>Tt\}}\sim
T^{\rho+\alpha}\frac{h(t)}{\mmp\{\xi>t\}},\quad t\to\infty.
\]
Sending
$t\to\infty$ we obtain, for every fixed $s>0$,
\[
\limsup_{t\to\infty}\frac{\mmp\{\xi>t\}}{h(t)}\sup_{u\in
[0,T]}\big|
\widetilde{Y}(ut)\big|\leq T^{\rho+\alpha}\sup_{u\in[s,\infty)} \biggl\llvert
\frac{\mmp\{\xi>u\}}{h(u)}\widetilde{Y}(u) \biggr\rrvert .
\]
Sending
now $s\to\infty$ shows that \eqref{eq:proof2} implies \eqref
{eq:proof1}. Let us first check that \eqref{eq:proof2} holds along the
arithmetic sequence $(n\delta)_{n\in\N}$. According to the
Borel--Cantelli lemma and Markov's inequality it suffices to check that
for some $l\in\N$
\begin{equation}
\label{eq:bc_integres} \sum_{n=1}^{\infty} \biggl(
\frac{\mmp\{\xi>n\delta\}}{h(\delta n)} \biggr)^{2l}\E\bigl[\widetilde{Y}(\delta
n)^{2l}\bigr]<\infty.
\end{equation}
To check \eqref{eq:bc_integres} we apply the Burkholder--Davis--Gundy
inequality in the form given in Theorem 11.3.2 of %in
\cite{Chow+Teicher:1997}, to obtain
\begin{align}
\label{eq:bdg_inequality} \E\bigl[\widetilde{Y}(t)^{2l}\bigr]&\leq K_l
\E \biggl[ \biggl(\sum_{k\geq0}\E \bigl(
\widehat{X}^2_{k+1}(t-S_k)\1_{\{S_k\leq t\}}|
\mathcal{F}_k \bigr) \biggr)^l \biggr]\notag
\\
&\quad+K_l\E \Bigl[\sup_{k\geq0} \bigl(
\widehat{X}^{2l}_{k+1}(t-S_k)\1_{\{
S_k\leq t\}}
\bigr) \Bigr],
\end{align}
for some constant $K_l>0$, where we recall the notation $\mathcal
{F}_k=\sigma((X_j,\xi_j):\break1\leq j\leq k)$.

\begin{proof}[Proof of (\ref{eq:bc_integres}) under assumptions of Theorem \ref
{thm:main1}] Using assumption (A4) we infer from \eqref{eq:bdg_inequality}:
\begin{align}
\E&\bigl[\widetilde{Y}(t)^{2l}\bigr]
\nonumber
\\
&\leq K_l \E \biggl[ \biggl(\sum_{k\geq0}v(t-S_k)
\1_{\{S_k\leq t\}} \biggr)^l \biggr]+K_l\E \biggl[\sum
_{k\geq0}\widehat{X}_{k+1}^{2l}(t-S_k)
\1_{\{
S_k\leq t\}} \biggr]
\nonumber
\\
&\leq K_l \E \biggl[ \biggl(\sum_{k\geq0}v(t-S_k)
\1_{\{S_k\leq t\}} \biggr)^l \biggr]+K_lC_l\E
\biggl[\sum_{k\geq0}v^l(t-S_k)
\1_{\{S_k\leq t\}
} \biggr]\label{eq:bdg_inequality1}.
\end{align}
If $\beta\geq0$, then $t\mapsto v^l(t)$ varies regularly with
non-negative index $l\beta$. Therefore, Lemma \ref{lem:moments}(i) yields
\begin{equation*}
\E \biggl(\sum_{k\geq0}v^l(t-S_k)
\1_{\{S_k\leq t\}} \biggr)=O \biggl(\frac
{v^l(t)}{\mmp\{\xi>t\}} \biggr),\quad t\to\infty.
\end{equation*}
If $\beta\in(-\alpha,0)$, pick $l\in\N$ such that $l\beta<-\alpha$.
Then $v^l(t)=O(\mmp\{\xi>t\})$, as $t\to\infty$, and Lemma \ref
{lem:moments}(iii) yields
\begin{equation*}
\E \biggl[\sum_{k\geq0}v^l(t-S_k)
\1_{\{S_k\leq t\}} \biggr]=O(1),\quad t\to\infty.
\end{equation*}
Hence, in any case
\begin{equation}
\label{eq:proof_aux1} \E \biggl[\sum_{k\geq0}v^l(t-S_k)
\1_{\{S_k\leq t\}} \biggr]=O \biggl(\frac
{v^l(t)}{\mmp\{\xi>t\}} \biggr)+O(1),\quad t\to
\infty.
\end{equation}
To bound the first summand in \eqref{eq:bdg_inequality1} apply Lemma
\ref{lem:moments}(i) to obtain
\[
\E \biggl[ \biggl(\sum_{k\geq0}v(t-S_k)
\1_{\{S_k\leq t\}} \biggr)^l \biggr]=O \biggl( \biggl(
\frac{v(t)}{\mmp\{\xi>t\}} \biggr)^l \biggr),\quad t\to \infty.
\]
Combining this estimate with \eqref{eq:proof_aux1}, we see that \eqref
{eq:bc_integres} holds if we pick $l>(2\rho+\alpha-\beta)^{-1}$. This
proves \eqref{eq:bc_integres} under assumptions of Theorem \ref{thm:main1}.
\end{proof}

\begin{proof}[Proof of (\ref{eq:bc_integres}) under assumptions of Theorem \ref
{thm:main2}] From \eqref{eq:bdg_inequality} and using \eqref
{eq:thm2_main1} we have
\begin{align*}
\E\bigl[\widetilde{Y}(t)^{2l}\bigr]\leq K_l
C_1^{l}\E \biggl[ \biggl(\sum
_{k\geq
0}h(t-S_k)\1_{\{S_k\leq t\}}
\biggr)^l \biggr]+K_lC_l\E \biggl[\sum_{k\geq0}h(t-S_k)
\1_{\{S_k\leq t\}} \biggr].
\end{align*}
\xch{Lemma \ref{lem:moments}(i)}{Lemma \eqref{lem:moments}(i)} gives %provides
us the estimate
\[
\E\bigl[\widetilde{Y}(t)^{2l}\bigr]=O \biggl( \biggl(
\frac{h(t)}{\mmp\{\xi>t\}} \biggr)^{l} \biggr),\quad t\to\infty.
\]
Therefore, \eqref{eq:bc_integres} holds if we choose $l\in\N$ such that
$l(\alpha+\rho)>1$. This proves \eqref{eq:bc_integres} under the
assumptions of
Theorem \ref{thm:main2}.

It remains to show that
\begin{align*}
\frac{\mmp\{\xi>n\delta\}}{h(n\delta)}\sup_{t\in[n\delta,(n+1)\delta
)} \bigg|&\sum
_{k\geq0}\bigl(\widehat{X}_{k+1}\bigl((n+1)
\delta-S_k\bigr)\1_{\{S_k\leq
(n+1)\delta\}}
\\
&-\widehat{X}_{k+1}(t-S_k)\1_{\{S_k\leq t\}}\bigr) \bigg|
\overset{a.s.} {\to} 0,
\end{align*}
as $n\to\infty$, which in turn is an obvious consequence of regular
variation of $t\mapsto\mmp\{\xi>t\}/h(t)$ and
\begin{equation}
\label{eq:proof_aux3} \frac{\mmp\{\xi>n\}}{h(n)}\sum_{k\geq0}V_{k+1}(n
\delta-S_k)\1_{\{
S_k\leq n\delta\}}\overset{a.s.} {\to} 0,\quad n\to\infty,
\end{equation}
where $V_{k+1}(t):=\sup_{y\in[0,\delta)}|\widehat{X}_{k+1}(t)-\widehat
{X}_{k+1}(t-y)\1_{\{y\leq t\}}|$.
\end{proof}

\begin{proof}[Proof of (\ref{eq:proof_aux3}) under assumptions of Theorem \ref
{thm:main1}] Applying Lemma \ref{lem:moments_rpwi}(i) with
$b(t)=t^{\rho-\varepsilon}$ and appropriate $\varepsilon>0$ we obtain
from (A5) that
\[
\E \biggl[ \biggl(\sum_{k\geq0}V_{k+1}(t-S_k)
\1_{\{S_k\leq t\}} \biggr)^{l} \biggr]=O \biggl( \biggl(
\frac{t^{\rho-\varepsilon}}{\mmp\{\xi>t\}
} \biggr)^l \biggr),\quad t\to\infty.
\]
Hence \eqref{eq:proof_aux3} holds in view of the Borel--Cantelli lemma
and Markov's inequality, since
\begin{align*}
\sum_{n=1}^{\infty}\mmp \biggl\{
\frac{\mmp\{\xi>n\}}{h(n)}\sum_{k\geq
0}V_{k+1}(n
\delta-S_k)\1_{\{S_k\leq n\delta\}}>\varepsilon \biggr\}
\leq\widehat{C}\sum_{n=1}^{\infty}
\bigl(n^{\rho-\varepsilon
}{h(n)} \bigr)^{l}<\infty,
\end{align*}
for all $l\in\N$ such that $\varepsilon l>1$ and some $\widehat
{C}=\widehat{C}_l>0$.
\end{proof}

\begin{proof}[Proof of (\ref{eq:proof_aux3}) under assumptions of Theorem \ref
{thm:main1}] If the function
\[
t\mapsto\E \Bigl[ \Bigl(\sup_{y\in[0,\delta)}\big|\widehat
{X}_{k+1}(t)-\widehat{X}_{k+1}(t-y)\1_{\{y\leq t\}}\big|
\Bigr)^{l} \Bigr]
\]
is directly Riemann integrable, then
\[
\E \biggl[ \biggl(\sum_{k\geq0}V_{k+1}(t-S_k)
\1_{\{S_k\leq t\}} \biggr)^{l} \biggr]=o(1),\quad t\to\infty
\]
by Lemma \ref{lem:moments_rpwi}(ii). Hence \eqref{eq:proof_aux3} holds
by the same reasoning as above after applying the Borel--Cantelli
lemma. If \eqref{eq:thm2_main2} holds, then the last centered formula
also holds with $O(1)$ in the right-hand side by Lemma \ref
{lem:moments_rpwi}(iii), whence \eqref{eq:proof_aux3}. This finishes
the proofs of Theorems \ref{thm:main1} and \ref{thm:main2}.
\end{proof}

\section*{Acknowledgments}
The work of A.~Marynych was supported by the Alexander von Humboldt
Foundation. We thank two anonymous referees for careful reading,
valuable comments and corrections of our numerous blunders.

\appendix

\section{Appendix}

\subsection{Moment convergence for renewal shot noise process}
\begin{lemma}\label{lem:moments}
Let $f:[0,\infty)\to\R$ be a locally bounded measurable function and
suppose that relation \eqref{eq:tail_xi_reg_var} holds for some $\alpha
\in(0,1)$.
\begin{itemize}
\item[(i)] Assume that
\[
f(t)\sim t^{\rho}\ell_{f}(t),\quad t\to\infty,
\]
for some $\rho>-\alpha$ and $\ell_f$ slowly varying at infinity. Let
$(J_{\alpha,\rho}(u))_{u\geq0}$ be a fractionally integrated inverse
stable subordinator defined in \eqref{eq:FIISS_def} (and below). Then,
for every $l\in\N$,
\begin{align}
\label{eq:moment_convergence} &\lim_{t\to\infty}\E \biggl[ \biggl(\frac{\mmp\{\xi>t\}}{f(t)}
\sum_{k\geq
0}f(t-S_k)\1_{\{S_k\leq t\}}
\biggr)^l \biggr] \notag\\
&\quad= \E \bigl(J_{\alpha,\rho
}(u) \bigr)^l
\notag\\
&\quad=\frac{l!}{(\varGamma(1-\alpha))^l}\prod_{j=1}^{l}
\frac{\varGamma(1+\rho
+(j-1)(\alpha+\rho))}{\varGamma(j(\alpha+\rho)+1)}.
\end{align}
\item[(ii)] If $f$ is directly Riemann integrable, then, for every $l\in
\N$,
\[
\E \biggl[ \biggl(\sum_{k\geq0}f(t-S_k)
\1_{\{S_k\leq t\}} \biggr)^l \biggr]=o(1),\quad t\to\infty.
\]
\item[(iii)] If $f(t)=O(\mmp\{\xi>t\})$, as $t\to\infty$, then, for
every $l\in\N$,
\[
\E \biggl[ \biggl(\sum_{k\geq0}f(t-S_k)
\1_{\{S_k\leq t\}} \biggr)^l \biggr]=O(1),\quad t\to\infty.
\]
\end{itemize}
\end{lemma}
\begin{proof}
The formula for the moments of fractionally integrated inverse stable
subordinator (the second equality in \eqref{eq:moment_convergence}) is
known, see for example (3.65) in \cite{Iksanov_book:2016} or (2.17) in
\cite{Iksanov+Marynych+Meiners:2014}.

\begin{proof}[Proof of (I)] In case $\rho\in(-\alpha,0]$ this result is just
Lemma 5.3 in \cite{Iksanov+Marynych+Meiners:2014}. A perusal of the
proof of the aforementioned lemma shows that without any modifications
the constraint $\rho\in(-\alpha,0]$ can be replaced by $\rho>-\alpha$.
\end{proof}

\begin{proof}[Proof of (II)] If $l=1$ and the distribution of $S_1$ is
non-lattice the claim follows from the classical key renewal theorem.
If $l=1$ and the distribution of $S_1$ is lattice, the claim still
holds, see the penultimate centered formula on p.~94 in \cite
{Iksanov+Marynych+Vatutin:2015}. In particular, this means
\begin{equation}
\label{eq:mom_lemma_aux1} 0\leq m_1(t):=\E \biggl[\sum
_{k\geq0} \bigl\llvert f(t-S_k) \bigr\rrvert
\1_{\{S_k\leq
t\}} \biggr]\leq M_1,\quad t\geq0,
\end{equation}
for some constant $M_1>0$. Applying formula (5.19) in \cite
{Iksanov+Marynych+Meiners:2014} we obtain
\begin{equation}
\label{eq:mom_lemma_aux2} m_l(t):=\E \biggl[ \biggl(\sum
_{k\geq0} \bigl\llvert f(t-S_k) \bigr\rrvert
\1_{\{S_k\leq
t\}} \biggr)^l \biggr]=\int_0^t
r_l(t-y){\rm d}U(y),
\end{equation}
where $U(y)=\sum_{k\geq0}\mmp\{S_k\leq y\}$, $y\geq0$ is the renewal
function and
\[
r_l(t)=\sum_{j=0}^{l-1}v_j
\bigl\llvert f(t) \bigr\rrvert ^{l-j}(t)m_j\xch{(t),}{(t)}
\]
for some real constants $v_j$. We proceed by induction. Assume that we know
\[
m_j(t)\to0,\quad t\to\infty,\quad j=1,\ldots,l-1,
\]
in \xch{particular,}{particular}
\[
0\leq m_j(t)\leq M_j,\quad t\geq0,\quad j=1,\ldots,l-1.
\]
Then
\[
\big|r_l(t)\big|\leq\sum_{j=0}^{l-1}M_{j}|v_j||f(t)|^{l-j},
\quad t\geq0,
\]
and the right-hand side is directly Riemann integrable. By the same
reasoning as in case $l=1$ we obtain
\[
m_l(t)\to0,\quad t\to\infty,
\]
by the key renewal theorem.
\end{proof}

\noindent\textbf{Proof of (III).} Again, let us consider the case $l=1$ first. Put
$Z(t):=t-S_{\nu(t)-1}$ and note that
\[
\E \biggl[\sum_{k\geq0}f(t-S_k)
\1_{\{S_k\leq t\}} \biggr]=\E g\bigl(Z(t)\bigr),
\]
where $g(t):=f(t)/\mmp\{\xi>t\}$. Since $g$ is bounded, we have $\E
g(Z(t))=O(1)$, as $t\to\infty$. For arbitrary $l\in\N$ the result
follows from \eqref{eq:mom_lemma_aux1} and \eqref{eq:mom_lemma_aux2} by
induction in the same vein as in the proof of part
(ii).
\end{proof}

In the next lemma we give an upper bound on the moments of random
process with immigration under assumption \eqref{eq:tail_xi_reg_var}.
Recall the notation
$Y(t)=\sum_{k\geq0}X_{k+1}(t-S_k)\1_{\{S_k\leq t\}}$.
\begin{lemma}\label{lem:moments_rpwi}
Assume that \eqref{eq:tail_xi_reg_var} holds for some $\alpha\in(0,1)$.
\begin{itemize}
\item[(i)] Suppose there exists a locally bounded measurable function
$b:[0,\infty)\to[0,\infty)$ such that
\[
b(t)\sim t^{\beta}\ell_b(t),\quad t\to\infty,
\]
for some $\beta>-\alpha$ and $\ell_b$ slowly varying at infinity. If
for every $l\in\N$
\[
\E\bigl[\big|X(t)\big|^{l}\bigr]\leq b^l(t),\quad t\geq0,
\]
then for every $l\in\N$ we have
\begin{equation}
\label{eq:moments_rpwi_gen_bound1} \E\bigl[\big|Y(t)\big|^l\bigr] = O \biggl( \biggl(
\frac{b(t)}{\mmp\{\xi>t\}} \biggr)^l \biggr),\quad t\to\infty.
\end{equation}
\item[(ii)] Suppose that for every $l\in\N$ there exists a directly
Riemann integrable function $b_l:[0,\infty)\to[0,\infty)$ such that
\[
\E\bigl[\big|X(t)\big|^{l}\bigr]\leq b_l(t),\quad t\geq0.
\]
Then, for every $l\in\N$
\begin{equation}
\label{eq:moments_rpwi_gen_bound2} \E\bigl[\big|Y(t)\big|^l\bigr] = o(1),\quad t\to\infty.
\end{equation}
\item[(iii)] Suppose that for every fixed $l\in\N$ we have
\[
\E\bigl[\big|X(t)\big|^{l}\bigr]= O\bigl(\mmp\{\xi>t\}\bigr),\quad t\to
\infty.
\]
Then, for every $l\in\N$
\begin{equation}
\label{eq:moments_rpwi_gen_bound3} \E\bigl[\big|Y(t)\big|^l\bigr] = O(1),\quad t\to\infty.
\end{equation}
\end{itemize}
\end{lemma}
\begin{proof}
Put $a_l(t):=\E[|X(t)|^l]$ for $l\in\N$ and
\[
Z(t):=\sum_{k\geq0}\big|X_{k+1}(t-S_k)\big|
\1_{\{S_k\leq t\}},\quad t\geq0.
\]
Clearly, $\E[|Y(t)|^l]\leq\E[[Z(t)]^l]$ for all $t\geq0$ and $l\in\N
$. We prove \eqref{eq:moments_rpwi_gen_bound1}, \eqref
{eq:moments_rpwi_gen_bound2} and \eqref{eq:moments_rpwi_gen_bound3}
with $\E[Z(t)^l]$ replacing $\E[|Y(t)|^l]$ in the left-hand sides.
From the definition of random process with immigration it follows that
\[
Z(t)\eqdistr\big|X(t)\big|+\widehat{Z}(t-\xi)\1_{\{\xi\leq t\}},\quad t\geq0,
\]
where $\widehat{Z}(t)\eqdistr Z(t)$ for every fixed $t\geq0$ and
$\widehat{Z}(t)$ is independent of $(X,\xi)$ in the right-hand side.
Taking expectations we obtain
\begin{equation}
\label{eq:rpwi_first_moment} \E\bigl[Z(t)\bigr] = a_1(t) + \E\bigl[Z(t-
\xi_1)\bigr]\1_{\{\xi\leq t\}},\quad t\geq0,
\end{equation}
whilst, for $l\geq2$, we have
\begin{align}
&\E\bigl[Z(t)^l\bigr]\notag\\
 &\quad= a_l(t)+\sum
_{j=1}^{l-1}\binom{l}{j}\E \bigl[\big|X(t)\big|^{l-j}
\bigl(\widehat{Z}(t-\xi)\bigr)^j\1_{\{\xi\leq t\}}\bigr]+\E\bigl[Z(t-
\xi)^l\1 _{\{\xi\leq t\}}\bigr]
\nonumber
\\
&\quad=a_l(t)+\sum_{j=1}^{l-1}
\binom{l}{j}\int_0^{\infty}\int
_0^t z^{l-j}\E \bigl[Z(t-y)^j
\bigr]\mmp\bigl\{\big|X(t)\big|\in{\rm d}z,\xi\in{\rm d}y\bigr\}
\nonumber
\\
&\qquad+\E\bigl[Z(t-\xi)^l\1_{\{\xi\leq t\}}\bigr]
\nonumber
\\
&\quad\leq a_l(t)+\sum_{j=1}^{l-1}
\binom{l}{j} a_{l-j}(t)\sup_{0\leq y\leq
t}\E
\bigl[Z(y)^j\bigr]+\E\bigl[Z(t-\xi)^l\1_{\{\xi\leq t\}}
\bigr]\label{eq:rpwi_higher_moment}.
\end{align}

\renewcommand{\theCase}{\Roman{Case}}
\begin{Case} From Lemma \ref{lem:moments}(i) and formula \eqref
{eq:rpwi_first_moment} using the inequality $a_1(t)\leq b(t)$, $t\geq
0$, we obtain
\[
\E\bigl[Z(t)\bigr] = O \biggl(\frac{b(t)}{\mmp\{\xi>t\}} \biggr),\quad t\to\infty.
\]
Thus, \eqref{eq:moments_rpwi_gen_bound1} holds for $l=1$. We proceed by
induction. Assume that for every $j=1,\ldots,l-1$ there exists $C_j>0$
such that
\[
\E\bigl[Z(t)^j\bigr]\leq C_j \biggl(\frac{b(t)}{\mmp\{\xi>t\}}
\biggr)^j,\quad t\geq0.
\]
This implies
\[
\sup_{0\leq y\leq t}\E\bigl[Z(y)^j\bigr]\leq
C_j\sup_{0\leq y\leq t} \biggl(\frac
{b(y)}{\mmp\{\xi>y\}}
\biggr)^j\sim C_j \biggl(\frac{b(t)}{\mmp\{\xi>t\}
}
\biggr)^j,
\]
where the last relation follows from the regular variation of $t\mapsto
b(t)/\mmp\{\xi>t\}$ with positive index $\beta+\alpha$. Hence, from
equation \eqref{eq:rpwi_higher_moment} and the inequalities $a_j(t)\leq
b^j(t)$, $t\geq0$, $j=1,\ldots,l-1$, we deduce
\[
\E\bigl[Z(t)^l\bigr] \leq C^{\prime}\frac{b^l(t)}{(\mmp\{\xi>t\})^{l-1}} + \E
\bigl[Z(t-\xi)^l\1_{\{\xi\leq t\}}\bigr],\quad t\geq0,
\]
for some $C^{\prime}=C^{\prime}_l>0$. Since $t\mapsto C^{\prime
}b^l(t)/(\mmp\{\xi>t\})^{l-1}$ is regularly varying with index $l(\beta
+\alpha)-\alpha>-\alpha$, Lemma \ref{lem:moments}(i) yields
\[
\E\bigl[Z(t)^l\bigr] = O \biggl( \biggl(\frac{b(t)}{\mmp\{\xi>t\}}
\biggr)^l \biggr),\quad t\to\infty.
\]
\end{Case}

\begin{Case} Arguing by induction as in the proof of case (i) we
see from formulae \eqref{eq:rpwi_first_moment} and \eqref
{eq:rpwi_higher_moment} that
\begin{equation*}
\E\bigl[Z(t)^l\bigr] \leq\widehat{b}^{\prime}_l(t)
+ \E\bigl[Z(t-\xi_1)^l\1_{\{\xi
\leq t\}}\bigr],\quad t
\geq0,
\end{equation*}
for a directly Riemann integrable function $\widehat{b}^{\prime}_l$.
The claim follows from the key renewal theorem.
\end{Case}

\begin{Case} For $l=1$ the claim follows from Lemma \ref
{lem:moments}(iii) and formula \eqref{eq:rpwi_first_moment}. Using
inductive argument once again we obtain from \eqref
{eq:rpwi_higher_moment} that
\begin{equation*}
\E\bigl[Z(t)^l\bigr] \leq C^{\prime\prime}\mmp\{\xi>t\} + \E
\bigl[Z(t-\xi_1)^l\1_{\{\xi
\leq t\}}\bigr],\quad t\geq0,
\end{equation*}
for some $C^{\prime\prime}=C^{\prime\prime}_l>0$ and the claim follows
from Lemma \ref{lem:moments}(iii).\qedhere
\end{Case}
\end{proof}


\begin{thebibliography}{99}

%b1 ###
\bibitem{Alsmeyer+Iksanov+Marynych:2016}
%
\begin{barticle}
\bauthor{\bsnm{Alsmeyer}, \binits{G.}},
\bauthor{\bsnm{Iksanov}, \binits{A.}},
\bauthor{\bsnm{Marynych}, \binits{A.}}:
\batitle{Functional limit theorems for the number of occupied boxes in the
{B}ernoulli sieve}.
\bjtitle{Stoch. Process. Appl.}
\bvolume{127}(\bissue{3}),
\bfpage{995}--\blpage{1017}
(\byear{2017}).
\bid{doi={10.1016/j.spa.2016.07.007}, mr={3605718}}
\end{barticle}
%
%
\OrigBibText
%
\begin{barticle}
\bauthor{\bsnm{Alsmeyer}, \binits{G.}},
\bauthor{\bsnm{Iksanov}, \binits{A.}},
\bauthor{\bsnm{Marynych}, \binits{A.}}:
\batitle{Functional limit theorems for the number of occupied boxes in the
{B}ernoulli sieve}.
\bjtitle{Stochastic Process. Appl.}
\bvolume{127}(\bissue{3}),
\bfpage{995}--\blpage{1017}
(\byear{2017})
\end{barticle}
%
\endOrigBibText
\bptok{structpyb}%
\endbibitem

%b2 ###
\bibitem{Alsmeyer+Iksanov+Meiners:2015}
%
\begin{barticle}
\bauthor{\bsnm{Alsmeyer}, \binits{G.}},
\bauthor{\bsnm{Iksanov}, \binits{A.}},
\bauthor{\bsnm{Meiners}, \binits{M.}}:
\batitle{Power and exponential moments of the number of visits and related
quantities for perturbed random walks}.
\bjtitle{J. Theor. Probab.}
\bvolume{28}(\bissue{1}),
\bfpage{1}--\blpage{40}
(\byear{2015}).
\bid{doi={10.1007/s10959-012-0475-7}, mr={3320959}}
\end{barticle}
%
%
\OrigBibText
%
\begin{barticle}
\bauthor{\bsnm{Alsmeyer}, \binits{G.}},
\bauthor{\bsnm{Iksanov}, \binits{A.}},
\bauthor{\bsnm{Meiners}, \binits{M.}}:
\batitle{Power and exponential moments of the number of visits and related
quantities for perturbed random walks}.
\bjtitle{J. Theoret. Probab.}
\bvolume{28}(\bissue{1}),
\bfpage{1}--\blpage{40}
(\byear{2015})
\end{barticle}
%
\endOrigBibText
\bptok{structpyb}%
\endbibitem

%b3 ###
\bibitem{Bingham+Goldie+Teugels:1987}
%
\begin{bbook}
\bauthor{\bsnm{Bingham}, \binits{N.H.}},
\bauthor{\bsnm{Goldie}, \binits{C.M.}},
\bauthor{\bsnm{Teugels}, \binits{J.L.}}:
\bbtitle{Regular Variation}.
\bsertitle{Encycl. Math. Appl.},
vol.~\bseriesno{27}.
\bpublisher{Cambridge University Press},
\blocation{Cambridge}
(\byear{1987}),
\bcomment{491~p.}
\bid{doi={10.1017/CBO9780511721434}, mr={0898871}}
\end{bbook}
%
%
\OrigBibText
%
\begin{bbook}
\bauthor{\bsnm{Bingham}, \binits{N.H.}},
\bauthor{\bsnm{Goldie}, \binits{C.M.}},
\bauthor{\bsnm{Teugels}, \binits{J.L.}}:
\bbtitle{Regular Variation}.
\bsertitle{Encyclopedia of Mathematics and its Applications},
vol.~\bseriesno{27}.
\bpublisher{Cambridge University Press, Cambridge}
(\byear{1987}),
\bcomment{491~p.}
\end{bbook}
%
\endOrigBibText
\bptok{structpyb}%
\endbibitem

%b4 ###
\bibitem{Chow+Teicher:1997}
%
\begin{bbook}
\bauthor{\bsnm{Chow}, \binits{Y.S.}},
\bauthor{\bsnm{Teicher}, \binits{H.}}:
\bbtitle{Probability Theory. Independence, Interchangeability, Martingales},
\bedition{3}rd edn.
\bsertitle{Springer Texts Statist.},
\bpublisher{Springer}
(\byear{1997}),
\bcomment{488~p.}
\end{bbook}
%
%
\OrigBibText
%
\begin{bbook}
\bauthor{\bsnm{Chow}, \binits{Y.S.}},
\bauthor{\bsnm{Teicher}, \binits{H.}}:
\bbtitle{Probability Theory. Independence, Interchangeability, Martingales},
\bedition{3}rd edn.
\bsertitle{Springer Texts in Statistics},
p.~\bfpage{488}.
\bpublisher{Springer}
(\byear{1997})
\end{bbook}
%
\endOrigBibText
\bptok{structpyb}%
\endbibitem

%b5 ###
\bibitem{Feller:1968}
%
\begin{bbook}
\bauthor{\bsnm{Feller}, \binits{W.}}:
\bbtitle{An Introduction to Probability Theory and Its Applications. {V}ol.
{II}}.
\bedition{2nd edn.},
\bpublisher{John Wiley \& Sons, Inc.},
\blocation{New York, London, Sydney}
(\byear{1971}),
\bcomment{669~p.}
\bid{mr={0270403}}
\end{bbook}
%
%
\OrigBibText
%
\begin{bbook}
\bauthor{\bsnm{Feller}, \binits{W.}}:
\bbtitle{An Introduction to Probability Theory and Its Applications. {V}ol.
{II}}.
\bsertitle{Second edition},
p.~\bfpage{669}.
\bpublisher{John Wiley \& Sons, Inc., New York-London-Sydney}
(\byear{1971})
\end{bbook}
%
\endOrigBibText
\bptok{structpyb}%
\endbibitem

%b6 ###
\bibitem{Iksanov:2013}
%
\begin{barticle}
\bauthor{\bsnm{Iksanov}, \binits{A.}}:
\batitle{Functional limit theorems for renewal shot noise processes with
increasing response functions}.
\bjtitle{Stoch. Process. Appl.}
\bvolume{123}(\bissue{6}),
\bfpage{1987}--\blpage{2010}
(\byear{2013}).
\bid{doi={10.1016/j.spa.2013.01.019}, mr={3038496}}
\end{barticle}
%
%
\OrigBibText
%
\begin{barticle}
\bauthor{\bsnm{Iksanov}, \binits{A.}}:
\batitle{Functional limit theorems for renewal shot noise processes with
increasing response functions}.
\bjtitle{Stoch. Process. Appl.}
\bvolume{123}(\bissue{6}),
\bfpage{1987}--\blpage{2010}
(\byear{2013})
\end{barticle}
%
\endOrigBibText
\bptok{structpyb}%
\endbibitem

%b7 ###
\bibitem{Iksanov_book:2016}
%
\begin{bbook}
\bauthor{\bsnm{Iksanov}, \binits{A.}}:
\bbtitle{Renewal Theory for Perturbed Random Walks and Similar Processes}.
\bsertitle{Probab. Appl.},
\bpublisher{Birkh\"{a}user}
(\byear{2016}),
\bcomment{250~p.}
\end{bbook}
%
%
\OrigBibText
%
\begin{bbook}
\bauthor{\bsnm{Iksanov}, \binits{A.}}:
\bbtitle{Renewal Theory for Perturbed Random Walks and Similar Processes}.
\bsertitle{Probability and Its Applications},
p.~\bfpage{250}.
\bpublisher{Birkh\"{a}user}
(\byear{2016})
\end{bbook}
%
\endOrigBibText
\bptok{structpyb}%
\endbibitem

%b8 ###
\bibitem{Iksanov+Jedidi+Bouzeffour:2017+}
%
\begin{botherref}
\oauthor{\bsnm{Iksanov}, \binits{A.}},
\oauthor{\bsnm{Jedidi}, \binits{W.}},
\oauthor{\bsnm{Bouzeffour}, \binits{F.}}:
Functional limit theorems for the number of busy servers in a
{G}/{G}/$\infty$
queue.
Preprint available at \url{https://arxiv.org/pdf/1610.08662.pdf}
\end{botherref}
%
%
\OrigBibText
%
\begin{botherref}
\oauthor{\bsnm{Iksanov}, \binits{A.}},
\oauthor{\bsnm{Jedidi}, \binits{W.}},
\oauthor{\bsnm{Bouzeffour}, \binits{F.}}:
Functional limit theorems for the number of busy servers in a
{G}/{G}/$\infty$
queue.
Preprint available at \url{https://arxiv.org/pdf/1610.08662.pdf}
\end{botherref}
%
\endOrigBibText
\bptok{structpyb}%
\endbibitem

%b9 ###
\bibitem{Iksanov+Kabluchko+Marynych:2016-1}
%
\begin{barticle}
\bauthor{\bsnm{Iksanov}, \binits{A.}},
\bauthor{\bsnm{Kabluchko}, \binits{Z.}},
\bauthor{\bsnm{Marynych}, \binits{A.}}:
\batitle{Weak convergence of renewal shot noise processes in the case
of slowly
varying normalization}.
\bjtitle{Stat. Probab. Lett.}
\bvolume{114},
\bfpage{67}--\blpage{77}
(\byear{2016}).
\bid{doi={10.1016/j.spl.2016.03.015}, mr={3491974}}
\end{barticle}
%
%
\OrigBibText
%
\begin{barticle}
\bauthor{\bsnm{Iksanov}, \binits{A.}},
\bauthor{\bsnm{Kabluchko}, \binits{Z.}},
\bauthor{\bsnm{Marynych}, \binits{A.}}:
\batitle{Weak convergence of renewal shot noise processes in the case
of slowly
varying normalization}.
\bjtitle{Statist. Probab. Lett.}
\bvolume{114},
\bfpage{67}--\blpage{77}
(\byear{2016})
\end{barticle}
%
\endOrigBibText
\bptok{structpyb}%
\endbibitem

%b10 ###
\bibitem{Iksanov+Marynych+Meiners:2014}
%
\begin{barticle}
\bauthor{\bsnm{Iksanov}, \binits{A.}},
\bauthor{\bsnm{Marynych}, \binits{A.}},
\bauthor{\bsnm{Meiners}, \binits{M.}}:
\batitle{Limit theorems for renewal shot noise processes with eventually
decreasing response functions}.
\bjtitle{Stoch. Process. Appl.}
\bvolume{124}(\bissue{6}),
\bfpage{2132}--\blpage{2170}
(\byear{2014}).
\bid{doi={10.1016/j.spa.2014.02.007}, mr={3188351}}
\end{barticle}
%
%
\OrigBibText
%
\begin{barticle}
\bauthor{\bsnm{Iksanov}, \binits{A.}},
\bauthor{\bsnm{Marynych}, \binits{A.}},
\bauthor{\bsnm{Meiners}, \binits{M.}}:
\batitle{Limit theorems for renewal shot noise processes with eventually
decreasing response functions}.
\bjtitle{Stoch. Process. Appl.}
\bvolume{124}(\bissue{6}),
\bfpage{2132}--\blpage{2170}
(\byear{2014})
\end{barticle}
%
\endOrigBibText
\bptok{structpyb}%
\endbibitem

%b11 ###
\bibitem{Iksanov+Marynych+Meiners:2017-1}
%
\begin{barticle}
\bauthor{\bsnm{Iksanov}, \binits{A.}},
\bauthor{\bsnm{Marynych}, \binits{A.}},
\bauthor{\bsnm{Meiners}, \binits{M.}}:
\batitle{Asymptotics of random processes with immigration {I}: Scaling limits}.
\bjtitle{Bernoulli}
\bvolume{23}(\bissue{2}),
\bfpage{1233}--\blpage{1278}
(\byear{2017}).
\bid{doi={10.3150/15-BEJ776}, mr={3606765}}
\end{barticle}
%
%
\OrigBibText
%
\begin{barticle}
\bauthor{\bsnm{Iksanov}, \binits{A.}},
\bauthor{\bsnm{Marynych}, \binits{A.}},
\bauthor{\bsnm{Meiners}, \binits{M.}}:
\batitle{Asymptotics of random processes with immigration {I}: scaling limits}.
\bjtitle{Bernoulli}
\bvolume{23}(\bissue{2}),
\bfpage{1233}--\blpage{1278}
(\byear{2017})
\end{barticle}
%
\endOrigBibText
\bptok{structpyb}%
\endbibitem

%b12 ###
\bibitem{Iksanov+Marynych+Meiners:2017-2}
%
\begin{barticle}
\bauthor{\bsnm{Iksanov}, \binits{A.}},
\bauthor{\bsnm{Marynych}, \binits{A.}},
\bauthor{\bsnm{Meiners}, \binits{M.}}:
\batitle{Asymptotics of random processes with immigration {II}:
Convergence to
stationarity}.
\bjtitle{Bernoulli}
\bvolume{23}(\bissue{2}),
\bfpage{1279}--\blpage{1298}
(\byear{2017})
\end{barticle}
%
%
\OrigBibText
%
\begin{barticle}
\bauthor{\bsnm{Iksanov}, \binits{A.}},
\bauthor{\bsnm{Marynych}, \binits{A.}},
\bauthor{\bsnm{Meiners}, \binits{M.}}:
\batitle{Asymptotics of random processes with immigration {II}:
convergence to
stationarity}.
\bjtitle{Bernoulli}
\bvolume{23}(\bissue{2}),
\bfpage{1279}--\blpage{1298}
(\byear{2017})
\end{barticle}
%
\endOrigBibText
\bptok{structpyb}%
\endbibitem

%b13 ###
\bibitem{Iksanov+Marynych+Vatutin:2015}
%
\begin{barticle}
\bauthor{\bsnm{Iksanov}, \binits{A.M.}},
\bauthor{\bsnm{Marynych}, \binits{A.V.}},
\bauthor{\bsnm{Vatutin}, \binits{V.A.}}:
\batitle{Weak convergence of finite-dimensional distributions of the
number of
empty boxes in the {B}ernoulli sieve}.
\bjtitle{Theory Probab. Appl.}
\bvolume{59}(\bissue{1}),
\bfpage{87}--\blpage{113}
(\byear{2015}).
\bid{doi={10.1137/S0040585X97986904}, mr={3416065}}
\end{barticle}
%
%
\OrigBibText
%
\begin{barticle}
\bauthor{\bsnm{Iksanov}, \binits{A.M.}},
\bauthor{\bsnm{Marynych}, \binits{A.V.}},
\bauthor{\bsnm{Vatutin}, \binits{V.A.}}:
\batitle{Weak convergence of finite-dimensional distributions of the
number of
empty boxes in the {B}ernoulli sieve}.
\bjtitle{Theory Probab. Appl.}
\bvolume{59}(\bissue{1}),
\bfpage{87}--\blpage{113}
(\byear{2015})
\end{barticle}
%
\endOrigBibText
\bptok{structpyb}%
\endbibitem

%b14 ###
\bibitem{Iksanov+Kabluchko+Marynych+Shevchenko:2017}
%
\begin{barticle}
\bauthor{\bsnm{Iksanov}, \binits{A.}},
\bauthor{\bsnm{Kabluchko}, \binits{Z.}},
\bauthor{\bsnm{Marynych}, \binits{A.}},
\bauthor{\bsnm{Shevchenko}, \binits{G.}}:
\batitle{Fractionally integrated inverse stable subordinators}.
\bjtitle{Stoch. Process. Appl.}
\bvolume{127}(\bissue{1}),
\bfpage{80}--\blpage{106}
(\byear{2016})
\end{barticle}
%
%
\OrigBibText
%
\begin{barticle}
\bauthor{\bsnm{Iksanov}, \binits{A.}},
\bauthor{\bsnm{Kabluchko}, \binits{Z.}},
\bauthor{\bsnm{Marynych}, \binits{A.}},
\bauthor{\bsnm{Shevchenko}, \binits{G.}}:
\batitle{Fractionally integrated inverse stable subordinators}.
\bjtitle{Stochastic Process. Appl.}
\bvolume{127}(\bissue{1}),
\bfpage{80}--\blpage{106}
(\byear{2016})
\end{barticle}
%
\endOrigBibText
\bptok{structpyb}%
\endbibitem

%b15 ###
\bibitem{Kluppelberg+Kuhn:2004}
%
\begin{barticle}
\bauthor{\bsnm{Kl\"uppelberg}, \binits{C.}},
\bauthor{\bsnm{K\"uhn}, \binits{C.}}:
\batitle{Fractional {B}rownian motion as a weak limit of {P}oisson shot noise
processes---with applications to finance}.
\bjtitle{Stoch. Process. Appl.}
\bvolume{113}(\bissue{2}),
\bfpage{333}--\blpage{351}
(\byear{2004})
\end{barticle}
%
%
\OrigBibText
%
\begin{barticle}
\bauthor{\bsnm{Kl\"uppelberg}, \binits{C.}},
\bauthor{\bsnm{K\"uhn}, \binits{C.}}:
\batitle{Fractional {B}rownian motion as a weak limit of {P}oisson shot noise
processes---with applications to finance}.
\bjtitle{Stochastic Process. Appl.}
\bvolume{113}(\bissue{2}),
\bfpage{333}--\blpage{351}
(\byear{2004})
\end{barticle}
%
\endOrigBibText
\bptok{structpyb}%
\endbibitem

%b16 ###
\bibitem{Marynych:2015}
%
\begin{barticle}
\bauthor{\bsnm{Marynych}, \binits{A.}}:
\batitle{A note on convergence to stationarity of random processes with
immigration}.
\bjtitle{Theory Stoch. Process.}
\bvolume{20(36)}(\bissue{1}),
\bfpage{84}--\blpage{100}
(\byear{2016}).
\bid{mr={3502397}}
\end{barticle}
%
%
\OrigBibText
%
\begin{barticle}
\bauthor{\bsnm{Marynych}, \binits{A.}}:
\batitle{A note on convergence to stationarity of random processes with
immigration}.
\bjtitle{Theory Stoch. Process.}
\bvolume{20(36)}(\bissue{1}),
\bfpage{84}--\blpage{100}
(\byear{2016})
\end{barticle}
%
\endOrigBibText
\bptok{structpyb}%
\endbibitem

%b17 ###
\bibitem{Mikosch+Resnick:2006}
%
\begin{barticle}
\bauthor{\bsnm{Mikosch}, \binits{T.}},
\bauthor{\bsnm{Resnick}, \binits{S.}}:
\batitle{Activity rates with very heavy tails}.
\bjtitle{Stoch. Process. Appl.}
\bvolume{116}(\bissue{2}),
\bfpage{131}--\blpage{155}
(\byear{2006})
\end{barticle}
%
%
\OrigBibText
%
\begin{barticle}
\bauthor{\bsnm{Mikosch}, \binits{T.}},
\bauthor{\bsnm{Resnick}, \binits{S.}}:
\batitle{Activity rates with very heavy tails}.
\bjtitle{Stochastic Process. Appl.}
\bvolume{116}(\bissue{2}),
\bfpage{131}--\blpage{155}
(\byear{2006})
\end{barticle}
%
\endOrigBibText
\bptok{structpyb}%
\endbibitem

%b18 ###
\bibitem{Resnick+Rootzen:2000}
%
\begin{barticle}
\bauthor{\bsnm{Resnick}, \binits{S.}},
\bauthor{\bsnm{Rootz\'en}, \binits{H.}}:
\batitle{Self-similar communication models and very heavy tails}.
\bjtitle{Ann. Appl. Probab.}
\bvolume{10}(\bissue{3}),
\bfpage{753}--\blpage{778}
(\byear{2000})
\end{barticle}
%
%
\OrigBibText
%
\begin{barticle}
\bauthor{\bsnm{Resnick}, \binits{S.}},
\bauthor{\bsnm{Rootz\'en}, \binits{H.}}:
\batitle{Self-similar communication models and very heavy tails}.
\bjtitle{Ann. Appl. Probab.}
\bvolume{10}(\bissue{3}),
\bfpage{753}--\blpage{778}
(\byear{2000})
\end{barticle}
%
\endOrigBibText
\bptok{structpyb}%
\endbibitem

\end{thebibliography}
\end{document}